\documentclass[12pt,a4paper]{article}
\usepackage{amssymb}
\usepackage{amsmath}
\usepackage{latexsym}
\usepackage{amsthm}

\newcommand{\V}{\textup{Vol}(4_1)}
\newcommand{\F}{J_{4_1,0}}

\newtheorem{lem}{Lemma}
\newtheorem{thm}{Theorem}
\newtheorem*{thm*}{Theorem}

\newtheorem{prop}[thm]{Proposition}
\theoremstyle{definition}

\usepackage{xpatch}
\xpatchcmd{\proof}{\itshape}{\normalfont\proofnameformat}{}{}
\newcommand{\proofnameformat}{}

\newcommand\blfootnote[1]{%
  \begingroup
  \renewcommand\thefootnote{}\footnote{#1}%
  \addtocounter{footnote}{-1}%
  \endgroup
}

\begin{document}

\renewcommand{\proofnameformat}{\bfseries}

\begin{center}
{\Large\textbf{Quantum invariants of hyperbolic knots and extreme values of trigonometric products}}


\vspace{5mm}

\textbf{Christoph Aistleitner$^1$ and Bence Borda$^{1,2}$}

\vspace{5mm}

{\footnotesize $^1$Graz University of Technology

Institute of Analysis and Number Theory

Steyrergasse 30, 8010 Graz, Austria

\vspace{5mm}

$^2$Alfr\'ed R\'enyi Institute of Mathematics

Re\'altanoda utca 13--15, 1053 Budapest, Hungary

Email: \texttt{aistleitner@math.tugraz.at} and \texttt{borda@math.tugraz.at}}
\end{center}

\begin{abstract}
In this paper we study the relation between the function $J_{4_1,0}$, which arises from a quantum invariant of the figure-eight knot, and Sudler's trigonometric product. We find $J_{4_1,0}$ up to a constant factor along continued fraction convergents to a quadratic irrational, and we show that its asymptotics deviates from the universal limiting behavior that has been found by Bettin and Drappeau in the case of large partial quotients. We relate the value of $J_{4_1,0}$ to that of Sudler's trigonometric product, and establish asymptotic upper and lower bounds for such Sudler products in response to a question of Lubinsky.
\end{abstract}

\section{Introduction}

Quantum knot invariants arise in theoretical quantum physics, where a knot can be regarded as the spacetime orbit of a charged particle. A typical example of such an invariant is the $n$-colored Jones polynomial of the knot $K=4_1$ (the figure-eight knot), which is given by\footnote{Throughout the paper, empty sums equal $0$, and empty products equal $1$.}
\[ J_{4_1,n} (q) = \sum_{N=0}^\infty q^{-nN} \prod_{j=1}^N (1 - q^{n-j}) (1 - q^{n+j}), \qquad n \ge 2, \]
defined for roots of unity $q$. For a fixed $q$, the mapping $n \mapsto J_{4_1,n}(q)$ is periodic in $n$, and so the definition can be extrapolated backwards in $n$ to give
\begin{equation}\label{Fq}
\F(q)=\sum_{N=0}^{\infty} |(1-q)(1-q^2) \cdots (1-q^N)|^2
\end{equation}
for a root of unity $q$. Note that both sums actually have only finitely many terms for any root of unity $q$. The figure-eight knot is the simplest hyperbolic knot; for other hyperbolic knots $K$ one obtains formulas for $J_{K,0}$ which are of a somewhat similar but more complicated nature. The so-called Kashaev invariant $\langle K \rangle_n = J_{K,n} (e(1/n))$, $n=1,2,\dots$ is another quantum invariant of the knot $K$; here and for the rest of the paper $e(x)=e^{2 \pi i x}$. The Kashaev invariant plays a key role in the volume conjecture, an open problem in knot theory which relates quantum invariants of knots with the hyperbolic geometry of knot complements. For more general background information, see \cite{MY}; in the context of our present paper we refer to \cite{BD} and the references therein.

The functions $J_{K,0}$ also have an interpretation as quantum modular forms as introduced by Zagier \cite{ZA}, and are predicted by Zagier's modularity conjecture to satisfy an approximate modularity property. For the figure-eight knot the modularity conjecture has been established \cite{BD,GZ}; that is, the function $\F(q)$ satisfies a remarkable modularity relation of the form $\F(e(\gamma r))/\F(e(r)) \sim \varphi_{\gamma} (r)$, where $\gamma \in \mathrm{SL}_2(\mathbb{Z})$ acts on rational numbers $r$ in a natural way, and the asymptotics holds as $r \to \infty$ along rational numbers with bounded denominators. The ratio $\F(e(\gamma r))/\F(e(r))$ as a function of $r$ in general has a jump at every rational point, and consequently the asymptotics of $\F(e(a/b))$ along rationals $a/b$ with $b \to \infty$ is quite involved. It is known \cite{AH} that
\[ \F(e(1/n)) \sim \frac{n^{3/2}}{\sqrt[4]{3}} \exp \left( \frac{\V}{2 \pi} n \right) \qquad \textrm{as } n \to \infty , \]
where
\[ \V = 4 \pi \int_0^{5/6} \log \left( 2 \sin (\pi x) \right) \, \mathrm{d}x \approx 2.02988 \]
is the hyperbolic volume of the complement of the figure-eight knot; this follows from the fact that the volume conjecture, as well as its stronger form, the arithmeticity conjecture have been verified for $4_1$. Bettin and Drappeau \cite[Theorem 3]{BD} found the asymptotics of $\F(e(a/b))$ for more general rationals $a/b$ in terms of their continued fraction expansions: if $a/b=[a_0;a_1, \dots, a_k]$, $a_k>1$, is a sequence of rational numbers such that $(a_1+\cdots +a_k)/k \to \infty$, then
\begin{equation}\label{BettinDrappeau}
\log \F(e(a/b)) \sim \frac{\V}{2 \pi} \left( a_1 + \cdots + a_k \right) .
\end{equation}
The result applies to a large class of rationals, including $1/n=[0;n]$, as well as to almost all reduced fractions with denominator at most $n$, as $n \to \infty$. Verifying a conjecture made by Bettin and Drappeau, in this paper we will show that \eqref{BettinDrappeau} in general fails to be true without the assumption $(a_1+\cdots +a_k)/k \to \infty$.

The individual terms in \eqref{Fq} can be expressed in terms of the so-called Sudler products, which are defined as
\begin{equation} \label{sudler_def}
P_N(\alpha):=\prod_{n=1}^N |2\sin (\pi n \alpha)|, \qquad \alpha \in \mathbb{R}.
\end{equation}
This could also be written using $q$-Pochhammer symbols as
\[ P_N(\alpha) =| (q;q)_N| = |(1-q)(1-q^2) \cdots (1-q^N)| \qquad \text{with } q = e(\alpha) , \]
but \eqref{sudler_def} seems to be the more common notation. The history of such products goes back at least to work of Erd\H os and Szekeres \cite{ESZ} and Sudler \cite{sudler} around 1960,  and they seem to arise in many different contexts (see \cite{LU} for references). Bounds for such products also play a role in the solution of the Ten Martini Problem by Avila and Jitomirskaya \cite{AJ}. It is somewhat surprising that, despite the obvious connection between \eqref{Fq} and \eqref{sudler_def}, we have not found a reference where both objects appear together. A possible explanation is that \eqref{Fq} is only well-defined when $q = e(a/b)$ with $a/b$ being a rational, while the asymptotic order of \eqref{sudler_def} as $N \to \infty$ is only interesting when $\alpha$ is irrational. We will come back to this issue in Proposition~\ref{transferprinciple} below. 

Note that for any rational number $a/b$ we have $P_N (a/b)=0$ whenever $N \ge b$; in particular, this means that $\F(e(a/b)) = \sum_{N=0}^{b-1} P_N (a/b)^2$. The asymptotics \eqref{BettinDrappeau} a fortiori holds for more general functionals of the sequence $(P_N(a/b))_{0 \le N <b}$. For instance, it is not difficult to see that under the same conditions as those of \eqref{BettinDrappeau} for any real $c>0$,
\begin{equation}\label{BettinDrappeauc}
\log \left( \sum_{N=0}^{b-1} P_N (a/b)^c \right)^{1/c} \sim \frac{\V}{4 \pi} \left( a_1 + \cdots + a_k \right) ,
\end{equation}
and also
\begin{equation}\label{BettinDrappeaumax}
\log \max_{0 \le N <b} P_N (a/b) \sim \frac{\V}{4 \pi} \left( a_1 + \cdots + a_k \right) .
\end{equation}
In particular, the maximal term in $\F(e(a/b)) = \sum_{N=0}^{b-1} P_N(a/b)^2$ is almost as large as the sum itself. For the sake of completeness, we formally deduce \eqref{BettinDrappeauc} and \eqref{BettinDrappeaumax} from \eqref{BettinDrappeau} in Section \ref{generalsection}.

It is natural to consider the same quantities along the sequence of convergents $p_k/q_k=[a_0;a_1, \dots, a_k]$ to a given irrational number $\alpha=[a_0;a_1,a_2, \dots]$. Formulas \eqref{BettinDrappeau}, \eqref{BettinDrappeauc} and \eqref{BettinDrappeaumax} give precise results for a large class of irrationals, but not when the sequence $a_k$ is bounded; our main result concerns this case. Recall that $\alpha$ is \textit{badly approximable} if and only if $a_k$ is bounded, and that $\alpha$ is a \textit{quadratic irrational} if and only if $a_k$ is eventually periodic; in particular, quadratic irrationals are badly approximable.
\begin{thm}\label{quadraticasymptotics} Let $\alpha$ be a quadratic irrational. For any real $c>0$ and any $k \ge 1$,
\[ \log \left( \sum_{N=0}^{q_k-1} P_N (p_k/q_k)^c \right)^{1/c} = K_c (\alpha ) k +O \left( \max \{ 1, 1/c \} \right) \]
and
\[ \log \max_{0 \le N < q_k} P_N (p_k/q_k) = K_{\infty} (\alpha ) k +O(1) \]
with some constants $K_c (\alpha), K_{\infty}(\alpha) >0$. The implied constants depend only on $\alpha$.
\end{thm}
\noindent In particular, we have
$$
\log \F (p_k/q_k) = 2K_2 (\alpha ) k +O(1) .
$$
The proof of Theorem \ref{quadraticasymptotics} is based on the self-similar structure of quadratic irrationals; that is, on the periodicity of the continued fraction. In fact, it is not difficult to construct a badly approximable $\alpha$ for which the result is not true.

Quadratic irrationals exhibit a remarkable deviation from the universal behavior of irrationals with unbounded partial quotients. In contrast to \eqref{BettinDrappeau}, \eqref{BettinDrappeauc} and \eqref{BettinDrappeaumax}, the constants $K_c(\alpha)$ and $K_{\infty}(\alpha)$ in Theorem \ref{quadraticasymptotics} are in general not equal to each other, or to $\V \overline{a}(\alpha)/(4 \pi)$; here $\overline{a}(\alpha) = \lim_{k \to \infty} (a_1+\cdots +a_k)/k$ denotes the average partial quotient, which is of course simply the average over a period in the continued fraction expansion. We have not been able to calculate the precise value of $K_c (\alpha)$ for any specific quadratic irrational and for any $c$; as far as we can say, this seems to be a difficult problem. The precise value of $K_\infty (\alpha)$ is known for some quadratic irrationals with very small partial quotients, as a consequence of results in \cite{ATZ}; for $\alpha$ with larger partial quotients, calculating $K_\infty (\alpha)$ precisely also seems to be difficult.

However, it is possible to give fairly good general upper and lower bounds for $K_c(\alpha)$ and $K_\infty (\alpha)$ in terms of the partial quotients. Recall that for any quadratic irrational $\alpha$, we have $\log q_k = \lambda (\alpha) k+O(1)$ with some constant $\lambda (\alpha )>0$; see Section \ref{quadraticsection} for a simple way of computing $\lambda (\alpha )$ from the continued fraction. We start with three simple observations. First, for any rational number $a/b$ with $(a,b)=1$, the identity\footnote{The last step follows e.g.\ from taking the limit as $x \to 1$ in the factorization $(x^b-1)/(x-1)=\prod_{j=1}^{b-1} (x-e(j/b))$.}
\begin{equation}\label{lastterm}
P_{b-1}(a/b) = \prod_{n=1}^{b-1} |1-e(na/b)| = \prod_{j=1}^{b-1} |1-e(j/b)| =b
\end{equation}
provides the trivial lower bound $\max_{0 \le N <b} P_N(a/b) \ge b$. This immediately yields
\begin{equation}\label{triviallowerbound}
K_{\infty} (\alpha ) \ge \lambda (\alpha ) .
\end{equation}
Second, for any rational number $a/b$ with $(a,b)=1$ and any real $c>0$, we also have
\begin{equation}\label{trivialbounds}
\left( \frac{1}{b} \sum_{N=0}^{b-1} P_N (a/b)^c \right)^{1/c} \le \max_{0 \le N <b} P_N (a/b) \le \left( \sum_{N=0}^{b-1} P_N (a/b)^c \right)^{1/c},
\end{equation}
which in turn shows that
\begin{equation}\label{KcKinftybounds}
K_c(\alpha) - \frac{\lambda(\alpha)}{c} \leq K_\infty(\alpha) \leq K_c(\alpha) .
\end{equation}
Finally, we establish an antisymmetry of the sequence $(P_N(a/b))_{0 \le N <b}$; we call it the ``reflection principle''. It is based on an observation which was already made in \cite{ATZ}. 
\begin{prop}\label{dualityprinciple} For any rational number $a/b$ with $(a,b)=1$, and any integer $0 \leq N < b$,
$$
\log P_N(a/b) + \log P_{b-N-1}(a/b) = \log b.
$$
\end{prop}
\noindent Note that Proposition \ref{dualityprinciple} immediately implies that
\[ \log \max_{0 \le N<b} P_N (a/b) + \log \min_{0 \le N<b} P_N (a/b) = \log b, \]
thus relating the largest with the smallest value of $P_N(a/b)$. In particular, all results for $\max_{0 \le N <b} P_N(a/b)$ have straightforward analogues for the minimum. As a nice further application of Proposition \ref{dualityprinciple}, we deduce the average value of $\log P_N (a/b)$ as
\begin{equation}\label{averagelogPN}
\frac{1}{b} \sum_{N=0}^{b-1} \log P_N (a/b) = \frac{\log b}{2} .
\end{equation}

From \eqref{triviallowerbound} it follows that $K_c(\alpha)$ and $K_{\infty} (\alpha)$ exceed $\V \overline{a}(\alpha )/(4 \pi)$ whenever the quadratic irrational $\alpha$ has relatively small partial quotients. For instance, $\frac{1+\sqrt{5}}{2}=[1;1,1,1,\dots]$, $\sqrt{2}=[1;2,2,2,\dots]$, $\frac{1+\sqrt{13}}{2}=[2;3,3,3,\dots]$, $\sqrt{5}=[2;4,4,4,\dots]$, $\frac{1+\sqrt{29}}{2}=[3;5,5,5,\dots]$ and $\sqrt{10}=[3;6,6,6,\dots]$ satisfy
\begin{equation}\label{Kinftylowerbound}
\begin{split} K_{\infty} \bigg( \frac{1+\sqrt{5}}{2} \bigg) &\ge \log \frac{1+\sqrt{5}}{2} \approx 0.4812, \qquad K_{\infty} (\sqrt{2}) \ge \log (1+\sqrt{2}) \approx 0.8814, \\ K_{\infty} \bigg( \frac{1+\sqrt{13}}{2} \bigg) &\ge \log \frac{3+\sqrt{13}}{2} \approx 1.1948, \qquad K_{\infty} (\sqrt{5}) \ge \log (2+\sqrt{5}) \approx 1.4436, \\ K_{\infty} \bigg( \frac{1+\sqrt{29}}{2}\bigg) &\ge \log \frac{5+\sqrt{29}}{2} \approx 1.6472, \qquad K_{\infty} (\sqrt{10}) \ge \log (3+\sqrt{10}) \approx 1.8184, \end{split}
\end{equation}
whereas
\[ \begin{split} \frac{\V}{4 \pi} \overline{a} \bigg( \frac{1+\sqrt{5}}{2} \bigg) &\approx 0.1615, \qquad \frac{\V}{4 \pi} \overline{a} (\sqrt{2}) \approx 0.3231, \\ \frac{\V}{4 \pi} \overline{a} \bigg( \frac{1+\sqrt{13}}{2} \bigg) &\approx 0.4846, \qquad \frac{\V}{4 \pi} \overline{a} (\sqrt{5}) \approx 0.6461, \\ \frac{\V}{4 \pi} \overline{a} \bigg( \frac{1+\sqrt{29}}{2} \bigg) &\approx 0.8077, \qquad \frac{\V}{4 \pi} \overline{a} (\sqrt{10}) \approx 0.9692 . \end{split} \]
In particular, the sequence of convergents to these quadratic irrationals violate \eqref{BettinDrappeau}, \eqref{BettinDrappeauc} and \eqref{BettinDrappeaumax}, demonstrating that the condition $(a_1+\cdots +a_k)/k \to \infty$ cannot be removed.

For a quadratic irrational $\alpha$ with large partial quotients, the constants $K_c (\alpha)$ and $K_{\infty} (\alpha )$ are nevertheless close to $\V \overline{a}(\alpha) /(4 \pi )$. Indeed, from results of Bettin and Drappeau \cite[Theorem 2 and Lemma 15]{BD} it follows that for any rational $a/b=[a_0;a_1, \dots, a_k]$, $a_k>1$,
\begin{equation}\label{BettinDrappeaugeneral}
\log \F(e(a/b)) = \frac{\V}{2 \pi} (a_1+\cdots +a_k) +O(A+k \log A)
\end{equation}
with $A=1+\max_{1 \le \ell \le k} a_{\ell}$ and a universal implied constant. A fortiori, for any real $c>0$,
\begin{equation}\label{BettinDrappeaugeneralc}
\log \left( \sum_{N=0}^{b-1} P_N(a/b)^c \right)^{1/c} = \frac{\V}{4 \pi} (a_1+\cdots +a_k) +O(A+k \max \{ 1, 1/c \} \log A)
\end{equation}
and
\begin{equation}\label{BettinDrappeaugeneralmax}
\log \max_{0 \le N <b} P_N (a/b) = \frac{\V}{4 \pi} (a_1+\cdots +a_k) +O(A+k \log A) .
\end{equation}
The last two relations immediately show that for any quadratic irrational $\alpha$ and any $0<c \le \infty$,
\begin{equation}\label{Kcestimate}
K_c (\alpha ) =  \frac{\V}{4 \pi} \overline{a}(\alpha ) +O( \max \{ 1, 1/c \} \log A(\alpha) )
\end{equation}
with $A(\alpha ) = 1+ \max_{k \ge 1} a_k$ and a universal implied constant.

In principle it might be the case that $K_\infty (\alpha )= K_c (\alpha )$ holds for $c$ beyond some threshold; note that this would be in accordance with the asymptotics \eqref{BettinDrappeauc} and \eqref{BettinDrappeaumax} for the case $(a_1 + \dots + a_k)/k \to \infty$. However, we rather believe that $K_\infty (\alpha)<K_c (\alpha )$ for all quadratic irrationals and all $c$. In this direction, from \eqref{averagelogPN} and the Jensen inequality applied with the convex function $e^{cx}$ we deduce that for any real $c>0$,
\[ K_c (\alpha ) \ge \left( \frac{1}{c} + \frac{1}{2} \right) \lambda (\alpha) ; \]
in light of \eqref{triviallowerbound} and \eqref{KcKinftybounds}, this is nontrivial for $0<c<2$. In particular, $K_{\infty} (\alpha) < K_c (\alpha )$ for all small enough $c$, and the set $\{ K_c (\alpha ) \, : \, c>0 \}$ is infinite for all quadratic irrationals.

As we mentioned earlier, previous results on the Sudler product concerned the asymptotics of $P_N(\alpha)$ as $N \to \infty$ with a given irrational $\alpha$, whereas $\F(e(a/b)) = \sum_{N=0}^{b-1} P_N(a/b)^2$ has been studied for rational $a/b$. To make these two types of results easier to compare, we prove the following simple ``transfer principle''.
\begin{prop}\label{transferprinciple} Let $\alpha=[a_0;a_1,a_2,\dots]$ be an irrational number with convergents $p_k/q_k=[a_0;a_1,\dots, a_k]$. For any integers $k \ge 1$ and $0 \le N < q_k$,
\[ \left| \log P_N (\alpha ) - \log P_N (p_k/q_k) \right| \ll \frac{\log A_k}{a_{k+1}} \]
with $A_k=1+\max_{1 \le \ell \le k} a_{\ell}$ and a universal implied constant.
\end{prop}
\noindent In particular, for any real $c>0$,
\[ \left| \log \left( \sum_{N=0}^{q_k-1} P_N(\alpha )^c \right)^{1/c} - \log \left( \sum_{N=0}^{q_k-1} P_N(p_k/q_k)^c \right)^{1/c} \right| \ll \frac{\log A_k}{a_{k+1}} \]
and
\[ \left| \log \max_{0 \le N <q_k} P_N (\alpha ) - \log \max_{0 \le N <q_k} P_N (p_k/q_k) \right| \ll \frac{\log A_k}{a_{k+1}} . \]
Applying the transfer principle to a quadratic irrational $\alpha$, Theorem \ref{quadraticasymptotics} can thus be restated in the irrational setting as
\begin{equation}\label{irrationalsetting}
\begin{split} \log \left( \sum_{N=0}^M P_N(\alpha )^c \right)^{1/c} &= \frac{K_c (\alpha )}{\lambda (\alpha)} \log M +O \left( \max \{ 1, 1/c \} \right) , \\ \log \max_{0 \le N \le M} P_N(\alpha ) &= \frac{K_{\infty}(\alpha )}{\lambda (\alpha)} \log M +O(1) , \end{split}
\end{equation}
where $\lambda (\alpha) >0$ is defined by $\log q_k = \lambda (\alpha)k+O(1)$, as before. For an arbitrary irrational $\alpha$ the reflection principle becomes, with the notation of Proposition \ref{transferprinciple},
\[ \log P_N (\alpha) + \log P_{q_k-N-1} (\alpha) = \log q_k + O \left( \frac{\log A_k}{a_{k+1}} \right) ; \]
in particular,
\begin{equation}\label{dualityirrational}
\log \max_{0 \le N <q_k} P_N (\alpha) + \log \min_{0 \le N <q_k} P_N (\alpha) = \log q_k +O \left( \frac{\log A_k}{a_{k+1}} \right) .
\end{equation}
For the sake of simplicity, we state all remaining results in the irrational setting only.

Erd\H{o}s and Szekeres \cite{ESZ} proved that $\liminf P_N (\alpha ) =0$ for almost all $\alpha$ in the sense of the Lebesgue measure, and asked whether the relation is actually true for all irrational $\alpha$. Lubinsky \cite{LU} gave more quantitative results on $P_N (\alpha)$ in terms of the continued fraction expansion of $\alpha$. The metric result of Erd\H{o}s and Szekeres was extended to a convergence/divergence type criterion, and it was also shown that $\liminf P_N(\alpha) =0$ whenever $\alpha$ has unbounded partial quotients. In addition, the results in the same paper imply
\[ \liminf_{N \to \infty} P_N (\alpha ) < \infty \qquad \textrm{and} \qquad \limsup_{N \to \infty} \frac{P_N (\alpha )}{N} >0 \]
for all irrational $\alpha$. Note that the relations in the previous line also follow from the identity \eqref{lastterm} and the transfer and reflection principles; in fact, the limsup relation is a far-reaching generalization of our trivial lower bound \eqref{triviallowerbound}. More recently, Grepstad, Kaltenb\"ock and Neum\"uller \cite{GKN} established the remarkable relation $\liminf P_N (\frac{1+\sqrt{5}}{2}) >0$, thus answering the question of Erd\H{o}s and Szekeres in the negative. On the other hand, however, in the paper \cite{GKN2} by the same authors it was shown that for the special irrational $\alpha = [0;a,a,a,\dots]=(\sqrt{a^2+4}-a)/2$ one has $\liminf P_N (\alpha)=0$ whenever $a$ is sufficiently large. Thus, rather remarkably, the question whether $\liminf P_N (\alpha)>0$ or $=0$ depends on the actual size of the partial quotients of $\alpha$ in a very sensitive way. Numerical experiments suggested a similar change of behavior for large values of $P_N$: for $\alpha= [0;a,a,a,\dots]$ it was conjectured that $\limsup P_N(\alpha)/N < \infty$ or $=\infty$, depending on the size of $a$. This problem was settled in \cite{ATZ}, where it was proved that for $\alpha= [0;a,a,a,\dots]$,
\begin{equation}\label{ale5}
\liminf_{N \to \infty} P_N (\alpha ) >0 \qquad \textrm{and} \qquad \limsup_{N \to \infty} \frac{P_N(\alpha )}{N} < \infty \qquad \textrm{if } a \le 5
\end{equation}
and
\begin{equation}\label{age6}
\liminf_{N \to \infty} P_N (\alpha ) =0 \qquad \textrm{and} \qquad \limsup_{N \to \infty} \frac{P_N(\alpha )}{N} = \infty \qquad \textrm{if } a \ge 6.
\end{equation}

Regarding general badly approximable irrationals $\alpha$, Lubinsky \cite{LU} proved that $|\log P_N (\alpha)| \ll \log N$; equivalently,
\begin{equation}\label{c1c2}
N^{-c_1} \ll P_N (\alpha ) \ll N^{c_2}
\end{equation}
with some constants $c_1, c_2$ and implied constants depending on $\alpha$. Let $c_1(\alpha)$ resp.\ $c_2(\alpha)$ denote the infimum of all $c_1$ resp.\ $c_2$ for which \eqref{c1c2} holds; Lubinsky remarked that it is an interesting problem to determine these constants. The reflection principle \eqref{dualityirrational} immediately shows that given a badly approximable $\alpha$ and a real constant $c$, we have $P_N (\alpha) \gg N^{-c} \Leftrightarrow P_N (\alpha ) \ll N^{1+c}$, with implied constants depending on $\alpha$. Therefore $c_2(\alpha)= c_1 (\alpha)+1$, which is a striking general relation that seems not to have been noticed so far. Thus establishing the optimal value of $c_1$ and that of $c_2$ in \eqref{c1c2} are actually one and the same problem.

For a quadratic irrational $\alpha$ our main result \eqref{irrationalsetting} shows that $c_2(\alpha ) =c_1(\alpha )+1= K_{\infty} (\alpha )/\lambda(\alpha)$, and in fact
\[ 0< \liminf_{N \to \infty} N^{c_1(\alpha)} P_N(\alpha) < \infty \quad \textrm{and} \quad 0< \limsup_{N \to \infty} \frac{P_N(\alpha )}{N^{c_2(\alpha )}} < \infty . \]
In particular,
\[ \liminf_{N \to \infty} P_N(\alpha )>0 \Longleftrightarrow \limsup_{N \to \infty} \frac{P_N(\alpha)}{N} < \infty \Longleftrightarrow K_{\infty}(\alpha ) = \lambda (\alpha ) , \]
where the last condition means that our trivial lower bound \eqref{triviallowerbound} holds with equality. Note that this relation also explains why the behavior of the liminf and the limsup in \eqref{ale5} and \eqref{age6} changes at the same critical value of $a$; this is also reflected in the six examples we gave in \eqref{Kinftylowerbound}, where we actually have equality everywhere except for $\sqrt{10}$, in which case the inequality is strict. It seems to be a difficult problem to give a complete characterization of all quadratic irrationals $\alpha$ such that $K_{\infty}(\alpha) = \lambda (\alpha)$; this is the subject of an upcoming paper of Grepstad, Neum\"uller and Zafeiropoulos \cite{GNZ}. 

The results in our paper allow us to give a fairly precise estimate for the constants $c_1(\alpha ), c_2(\alpha)$ in the question of Lubinsky. From \eqref{Kcestimate} with $c=\infty$ we obtain that for any quadratic irrational $\alpha$,
\[ c_2(\alpha ) = c_1(\alpha ) +1= \frac{\V}{4 \pi} \cdot \frac{\overline{a}(\alpha )}{\lambda (\alpha )} + O \left( \frac{\log A (\alpha)}{\lambda (\alpha )} \right) \]
with a universal implied constant. In particular, for $\alpha = [0;a,a,a,\dots]$ we have
\[ c_2(\alpha ) = c_1(\alpha ) +1  =  \frac{\V}{4 \pi} \cdot \frac{a}{\log \frac{a+\sqrt{a^2+4}}{2}} +O(1) = \frac{\V}{4 \pi} \cdot \frac{a}{\log a} +O(1) . \]

The discussion above shows that $K_{\infty} (\alpha) \ge K_{\infty} (\frac{1+\sqrt{5}}{2}) = \log \frac{1+\sqrt{5}}{2}$ for all quadratic irrationals. We do not know whether $K_c (\alpha) \ge K_c (\frac{1+\sqrt{5}}{2})$ for all quadratic irrationals and all $c>0$; in fact, we do not even know the precise value of $K_c (\frac{1+\sqrt{5}}{2})$ for any $c$. We mention, however, that numerical evidence found by Zagier \cite{ZA} and Bettin and Drappeau \cite{BD} suggests that for the golden mean we have $\log \F(e(p_k/q_k)) \approx 1.1 k$; that is, $K_2 (\frac{1+\sqrt{5}}{2}) \approx 0.55$.

In this context, it is an interesting question to characterize those values of $N$ for which particularly large resp.\ small values of $P_N(\alpha)$ occur. It is also interesting to estimate the relative number of indices which generate such values of $P_N (\alpha)$. This would shed some light on the relation between the numbers $K_c(\alpha)$ and $K_\infty(\alpha)$ in Theorem \ref{quadraticasymptotics} and \eqref{irrationalsetting}. Essentially, the problem is whether the sum $\sum_{N=0}^M P_N(\alpha)^c$ is dominated by a very small number of indices $N$ which produce particularly large values of $P_N$, or if there are enough such indices so that the full sum is of a significantly different asymptotic order than its maximal term. We plan to come back to all these questions in a future paper.

Finally, we mention a further open problem. In \cite{ZA} Zagier introduced the function $h(x) = \log \left(\F(e(x)) / \F(e(1/x)) \right)$. A conjecture of Zagier, established by Bettin and Drappeau in \cite{BD}, implies that $h$ has jumps at all rational points. Zagier also suggested that $h(x)$ is continuous at irrational values of $x$ (more precisely, since $h$ is formally only defined for rational $x$, the conjecture is that $h$ can be extended to all reals such that it is continuous at irrational values). Let $\alpha$ be a quadratic irrational whose continued fraction expansion is of the simple form $\alpha = [0;a,a,a,\dots]$, and let $p_k/q_k$ be its convergents. Then it is easily seen that $h(p_k/q_k) = \log \left(\F(e(p_k/q_k)) \right) - \log \left( \F(e(p_{k-1}/q_{k-1})) \right)$. Thus, while we cannot prove that $h(p_k/q_k)$ converges as $k \to \infty$, as a consequence of Theorem \ref{quadraticasymptotics} we can at least conclude that $K^{-1} \sum_{k=1}^K h(p_k/q_k)$ converges as $K \to \infty$. Note that if the error term in Theorem \ref{quadraticasymptotics} could be reduced from $O(1)$ to $o(1)$, then we could deduce that actually $h(p_k/q_k)$ itself converges as $k \to \infty$, so such a conclusion is just beyond the reach of our theorem. Finally, note that Theorem \ref{quadraticasymptotics} implies that if $h$ can indeed be continuously extended to $\alpha$, then the only possible value is $h(\alpha) = 2 K_2(\alpha)$ with $K_2 (\alpha)$ from Theorem \ref{quadraticasymptotics}. Thus our results can be seen as progress towards Zagier's problem, while it seems that there is still a long way to go for a full solution of the problem. From the discussion above, one might expect that the problem requires different approaches according to whether the partial quotients of $\alpha$ are large (say, as in the case $(a_1 + \cdots + a_k)/k \to \infty$), or small (say, bounded). The most difficult case could be the one when the partial quotients of $\alpha$ are small, but there is no particular structure such as periodicity. 

\section{General rationals and irrationals}\label{generalsection}

Recalling \eqref{trivialbounds}, to deduce \eqref{BettinDrappeauc} and \eqref{BettinDrappeaumax} from \eqref{BettinDrappeau} we only need to show that the condition $(a_1+\cdots +a_k)/k \to \infty$ implies that $\log b /(a_1+\cdots +a_k) \to 0$; indeed, this will show that for any $c>0$,
\[ \log \left( \sum_{N=0}^{b-1} P_N(a/b)^c \right)^{1/c} \sim \log \max_{0 \le N <b} P_N (a/b) . \]
From the recursion satisfied by the convergents we get $b \le (a_1+1)\cdots (a_k+1)$. Letting $\overline{a}_k=(a_1+\cdots +a_k)/k$, the AM--GM inequality gives
\[ \frac{\log b}{a_1+\cdots +a_k} \le \frac{\log ((a_1+1)\cdots (a_k+1))}{a_1+\cdots +a_k} \le \frac{\log (\overline{a}_k+1)}{\overline{a}_k} \to 0 , \]
and we are done. To deduce \eqref{BettinDrappeaugeneralc} and \eqref{BettinDrappeaugeneralmax} from \eqref{BettinDrappeaugeneral}, simply note that
\[ \log b \le \log (a_1+1) + \cdots +\log (a_k+1) = O(k \log A), \]
and hence \eqref{trivialbounds} shows that for any $c>0$,
\[ \log \left( \sum_{N=0}^{b-1} P_N(a/b)^c \right)^{1/c} = \log \max_{0 \le N < b} P_N (a/b) + O\left( \frac{k}{c} \log A \right) . \]

\subsection{The reflection principle}

\begin{proof}[Proof of Proposition \ref{dualityprinciple}] By the definition of Sudler products, for any integer $0 \le N <b$,
\begin{equation}\label{factorization}
P_N (a/b) \cdot \prod_{n=N+1}^{b-1} |2 \sin (\pi n a/b)| = P_{b-1} (a/b) .
\end{equation}
A simple reindexing shows that here
\[ \prod_{n=N+1}^{b-1} |2 \sin (\pi n a/b)| = \prod_{j=1}^{b-N-1} |2 \sin (\pi (b-j) a/b)| = P_{b-N-1} (a/b) . \]
As observed in \eqref{lastterm}, we also have $P_{b-1}(a/b)=b$. Hence \eqref{factorization} yields
\[ \log P_N (a/b) + \log P_{b-N-1} (a/b) = \log b, \]
as claimed.
\end{proof}

\subsection{The transfer principle}

Let $\alpha =[a_0;a_1,a_2,\dots]$ be an arbitrary irrational number with convergents $p_k/q_k=[a_0;a_1, \dots, a_k]$. Let $A_k =1+ \max_{1 \le \ell \le k} a_{\ell}$, and let $\| x \|$ denote the distance from a real number $x$ to the nearest integer. The sequence $q_k$ satisfies the recursion $q_k=a_k q_{k-1} + q_{k-2}$ with initial conditions $q_0=1$, $q_1=a_1$. Recall that for any $k \ge 1$ and any $0<n<q_k$, we have $\| n \alpha \| \ge \| q_{k-1} \alpha \|$. Further, if $k \ge 1$, or $k=0$ and $a_1>1$, then
\begin{equation}\label{||qkalpha||}
\frac{1}{q_{k+1}+q_k} < \| q_k \alpha \| < \frac{1}{q_{k+1}} .
\end{equation}

The main tool in the proof of the transfer principle is a bound on a cotangent sum proved by Lubinsky \cite[Theorem 4.1]{LU}, which states that for any $k \ge 1$ and any $0\le N<q_k$,
\begin{equation}\label{cotangentsumirrational}
\left| \sum_{n=1}^N \cot (\pi n \alpha ) \right| \le \left( 124+ 24 \log A_k \right) q_k .
\end{equation}
The same bound holds in the rational setting as well, i.e.

\begin{equation}\label{cotangentsumrational}
\left| \sum_{n=1}^N \cot (\pi n p_k/q_k ) \right| \le \left( 124+ 24 \log A_k \right) q_k .
\end{equation}
Indeed, we can apply \eqref{cotangentsumirrational} to a sequence of irrational $\alpha$'s converging to $p_k/q_k$, whose continued fraction expansions have initial segments identical to $p_k/q_k=[a_0;a_1, \dots, a_k]$. The same cotangent sum and various generalizations thereof in the rational setting have been studied recently in \cite{BD2}, and used in \cite{BD} to establish \eqref{BettinDrappeau}. Cotangent sums have a long history in analytic number theory; some of them are known to satisfy interesting reciprocity formulas, and they also appear in the Nyman--Beurling--B\'aez-Duarte approach to the Riemann hypothesis. See \cite{BC,BC2} for more details.

\begin{proof}[Proof of Proposition \ref{transferprinciple}] We consider the cases $q_k<200$ and $q_k \ge 200$ separately, starting with the former; the value $200$ is of course basically accidental.

First, assume that $q_k<200$. If $k=1$ and $a_1=1$, then $N=0$ and we are done; we may therefore assume that either $k \ge 2$, or $k=1$ and $a_1>1$. For any $0<n<q_k$ we thus have
\[ 2 \ge |2 \sin (\pi n \alpha )| \ge 4 \| n \alpha \| \ge 4 \| q_{k-1} \alpha \| \ge \frac{4}{q_k + q_{k-1}} > \frac{1}{100}, \]
and similarly
\begin{equation} \label{sin_dist}
2 \ge |2 \sin (\pi n p_k/q_k)| \ge 4 \| n p_k/q_k \| \ge \frac{4}{q_k} > \frac{1}{50} . 
\end{equation}
Consequently, $|\log P_N (\alpha)| \ll 1$ and $|\log P_N (p_k/q_k)| \ll 1$, and we are done provided $a_{k+1}$ is bounded. For large $a_{k+1}$ note that
\[ |\sin (\pi n \alpha ) - \sin (\pi n p_k/q_k)| \le \pi n |\alpha -p_k/q_k| \le \frac{\pi}{a_{k+1}} . \]
In particular, by \eqref{sin_dist} we have
\[ \left| \frac{\sin (\pi n \alpha )}{\sin (\pi n p_k/q_k)} -1 \right| < \frac{50 \pi}{a_{k+1}} , \]
and hence
\[ \left( 1-\frac{50 \pi}{a_{k+1}} \right)^{200} \le \frac{P_N (\alpha )}{P_N (p_k/q_k)} \le \left( 1+\frac{50 \pi}{a_{k+1}} \right)^{200} . \]
This finishes the proof in the case $q_k<200$.

Next, assume that $q_k \ge 200$. Observe that $q_k \ge q_{\ell} \ge a_{\ell}$ for all $1 \le \ell \le k$; in particular, $q_k \ge A_k-1$. From the assumption $q_k \ge 200$ we thus deduce $q_k \ge \sqrt{200(A_k-1)} \ge 10 \sqrt{A_k}$. Using a trigonometric identity, we can write
\begin{equation}\label{prod1+xn}
\frac{P_N (\alpha )}{P_N (p_k/q_k )} = \left| \prod_{n=1}^N \frac{\sin (\pi n \alpha )}{\sin (\pi n p_k/q_k )} \right| = \left| \prod_{n=1}^N (1+x_n+y_n) \right|,
\end{equation}
where
\[ \begin{split} x_n &= \cos (\pi n (\alpha -p_k/q_k )) -1, \\ y_n &= \sin (\pi n (\alpha -p_k/q_k)) \cot (n \pi p_k/q_k ) . \end{split} \]
Here
\[ \left| \alpha - \frac{p_k}{q_k} \right| < \frac{1}{q_k q_{k+1}} \le \frac{1}{q_k^2} \left( 1-\frac{1}{\frac{q_k}{q_{k-1}}+1} \right) \le \frac{1}{q_k^2} \left( 1-\frac{1}{A_k+1} \right) . \]
From the Taylor expansions of sine and cosine, for all $0<n<q_k$,
\[ |x_n| \le \frac{\pi^2 n^2}{2} \left| \alpha - \frac{p_k}{q_k} \right|^2 \le \frac{\pi^2}{2 q_k^2} \le \frac{1}{20 A_k} , \]
as well as
\[ \begin{split} \left| y_n \right| &\le \frac{|\sin (\pi n (\alpha -p_k/q_k ))|}{\sin (\pi /q_k )} \le \frac{\pi n |\alpha -p_k/q_k|}{\pi /q_k - \pi^3 /(6q_k^3)} \le \frac{1}{1-\pi^2 /(6q_k^2)} \left( 1-\frac{1}{A_k+1} \right) \\ &\le \frac{1}{1-\pi^2 /(600 A_k)} \left( 1-\frac{1}{A_k+1} \right) \le 1-\frac{3}{10 A_k} .  \end{split} \]
The previous two estimates give $|x_n+y_n| \le 1-1/(4 A_k)$; the point is that $x_n+y_n$ is bounded away from $-1$.

Observe that for any $|x| \le 1-1/(4A_k)$,
\[ e^{x-2x^2 \log (4A_k)} \le 1+x \le e^x . \]
Indeed, one readily verifies that the function $e^{-x+2x^2 \log (4A_k)} (1+x)$ attains its minimum on the interval $[-1+1/(4A_k), 1-1/(4A_k)]$ at $x=0$. Applying this estimate with $x=x_n+y_n$ in each factor of \eqref{prod1+xn}, we obtain
\begin{equation}\label{pnalpha/pnpkqk}
\begin{split} \frac{P_N (\alpha )}{P_N (p_k/q_k )} &= \exp \left( O \left( \left| \sum_{n=1}^N  (x_n+y_n) \right| + \sum_{n=1}^N (x_n^2+y_n^2) \log (4A_k) \right) \right) . \end{split}
\end{equation}
Note that the right-hand side of \eqref{pnalpha/pnpkqk} provides both an upper and a lower bound for the quotient on the left-hand side. Since
\[ |x_n| \le \frac{\pi^2 n^2}{2} \left| \alpha - \frac{p_k}{q_k} \right|^2 \le \frac{\pi^2}{2 a_{k+1}^2 q_k^2}, \]
the contribution of $x_n$ and $x_n^2$ in \eqref{pnalpha/pnpkqk} is negligible:
\[ \sum_{n=1}^N |x_n| + \sum_{n=1}^N x_n^2 \log (4A_k) \ll \frac{1}{a_{k+1}^2 q_k} + \frac{\log A_k}{a_{k+1}^4 q_k^3}. \]
From Lubinsky's bound on cotangent sums \eqref{cotangentsumrational}, summation by parts and
\[ \left| \sin (\pi (n+1) (\alpha -p_k/q_k)) - \sin (\pi n (\alpha -p_k/q_k)) \right| \le \pi |\alpha -p_k/q_k| \ll \frac{1}{a_{k+1}q_k^2}, \]
we obtain
\[ \left| \sum_{n=1}^N y_n \right| = \left| \sum_{n=1}^N \sin (\pi n (\alpha -p_k/q_k)) \cot (\pi n p_k/q_k) \right| \ll \frac{\log A_k}{a_{k+1}} . \]
Finally,
\[ \sum_{n=1}^N y_n^2 \log (4A_k) \ll \sum_{n=1}^N \frac{\log A_k}{a_{k+1}^2 q_k^2 \| np_k/q_k \|^2} \le \sum_{j=1}^{q_k-1} \frac{\log A_k}{a_{k+1}^2 q_k^2 \| j/q_k \|^2} \ll \frac{\log A_k}{a_{k+1}^2}, \]
since the integers $np_k$, $1 \le n \le N$ attain each nonzero residue class modulo $q_k$ at most once. Hence \eqref{pnalpha/pnpkqk} simplifies as
\[ \frac{P_N (\alpha )}{P_N (p_k/q_k )} = \exp \left( O \left( \frac{\log A_k}{a_{k+1}} \right) \right), \]
which proves the proposition.
\end{proof}

\section{Quadratic irrationals}\label{quadraticsection}

Fix a quadratic irrational $\alpha$. Throughout this section, constants and implied constants depend only on $\alpha$. The continued fraction expansion is of the form $\alpha =[a_0;a_1,\dots, a_s, \overline{a_{s+1}, \dots, a_{s+p}}]$, where the overline means period. As before, $p_k/q_k=[a_0;a_1,\dots, a_k]$ denotes the $k$-th convergent to $\alpha$; further, let $\delta_k=(-1)^k (q_k \alpha -p_k)$. The sequences $q_k$ and $p_k$ satisfy the same recursion; consequently, $\delta_k=-a_k \delta_{k-1}+\delta_{k-2}$ for all $k \ge 2$. If $k \ge 1$, or $k=0$ and $a_1>1$, then $\delta_k = \| q_k \alpha \|$.

For any $1 \le r \le p$, let
\[ T_r = \left( \begin{array}{cc} 0 & 1 \\ 1 & a_{s+r+p} \end{array} \right) \cdots \left( \begin{array}{cc} 0 & 1 \\ 1 & a_{s+r+1} \end{array} \right) . \]
The recursion $q_k = a_k q_{k-1} +q_{k-2}$ can be written in the form
\[ T_r^m \left( \begin{array}{c} q_{s+r-1} \\ q_{s+r} \end{array} \right) = \left( \begin{array}{c} q_{s+mp+r-1} \\ q_{s+mp+r} \end{array} \right), \qquad m=1,2,\dots \]
Observe that $\det T_r =(-1)^p$, and that $\mathrm{tr}\, T_r$ does not depend on $r$. Therefore the eigenvalues $\eta$ and $\mu$ of $T_r$ are the same for all $1 \le r \le p$. Since $q_k \to \infty$ exponentially fast, we have, say, $\eta>1$ and $\mu=(-1)^p/\eta$. Consequently, the recursions for $q_k$ and $\delta_k$ have solutions
\begin{equation}\label{recursionsolutions}
\begin{split} q_{s+mp+r} &= C_r \eta^m + D_r (-1)^{mp} \eta^{-m}, \\ \delta_{s+mp+r} &= E_r \eta^{-m} \end{split}
\end{equation}
with some constants $C_r, E_r >0$ and $D_r$, $1 \le r \le p$. In particular, $\log q_k = \lambda (\alpha)k +O(1)$ with $\lambda (\alpha) = \log \eta^{1/p}$.

\begin{lem}\label{deltaklemma} For any $k \ge 0$, we have $\kappa \delta_k \le \delta_{k+1} \le (1-\kappa) \delta_k$ with some constant $\kappa >0$.
\end{lem}

\begin{proof} We claim that $\delta_k \le (a_{k+2}+2) \delta_{k+1}$ for all $k \ge 0$. Indeed, if $k=0$, $a_1=1$ this can be verified ``by hand''; else, from \eqref{||qkalpha||} we obtain
\[ \delta_k < \frac{1}{q_{k+1}} \le \frac{a_{k+2}+2}{q_{k+2}+q_{k+1}} < (a_{k+2}+2) \delta_{k+1} . \]
On the other hand,
\[ \delta_{k+1} \le a_{k+2} \delta_{k+1} = \delta_k - \delta_{k+2} \le \delta_k - \frac{1}{a_{k+3}+2} \delta_{k+1}, \]
and hence $\delta_{k+1} \le \delta_k (a_{k+3}+2)/(a_{k+3}+3)$. The claim thus follows with $\kappa=1/(\max_{k \ge 1} a_k +3)$.
\end{proof}

\subsection{Perturbed Sudler products}\label{sec_pert}

The fundamental objects in the proof of Theorem \ref{quadraticasymptotics} are ``perturbed'' versions of the Sudler product defined as
\[ P_{q_k} (\alpha, x) := \prod_{n=1}^{q_k} |2 \sin (\pi (n\alpha +(-1)^k x/q_k))|, \qquad x \in \mathbb{R}. \]
Perturbed Sudler products were first introduced by Grepstad, Kaltenb\"ock and Neum\"uller \cite{GKN}, and have since been used in \cite{ATZ} and \cite{GNZ}. The relevance of these functions come from the Ostrowski expansion of integers, which we now recall. Any integer $N \ge 0$ can be written in a unique way in the form $N=\sum_{k=0}^{\infty} b_k q_k$, where $0 \le b_0 \le a_1-1$ and $0 \le b_k \le a_{k+1}$, $k \ge 1$ are integers satisfying the extra rule that $b_{k-1}=0$ whenever $b_k=a_{k+1}$. Of course, the series only has finitely many nonzero terms; more precisely, if $0 \le N <q_{k+1}$, then $b_{k+1}=b_{k+2}=\cdots =0$. Given an integer $N \ge 0$ with Ostrowski expansion $N=\sum_{k=0}^{\infty} b_k q_k$, let us introduce the notation
\[ \varepsilon_k (N) := q_k \sum_{\ell =k+1}^{\infty} (-1)^{k+\ell} b_{\ell} \delta_{\ell}, \]
where $\delta_\ell=(-1)^\ell (q_\ell \alpha -p_\ell)$ was already defined at the beginning of this section. Further, we shall write $f(x)=|2 \sin (\pi x)|$.
\begin{lem}\label{ostrowskilemma} For any integer $N \ge 0$ with Ostrowski expansion $N=\sum_{k=0}^{\infty} b_k q_k$,
\[ P_N (\alpha ) = \prod_{k=0}^{\infty} \prod_{b=0}^{b_k-1} P_{q_k} (\alpha, bq_k \delta_k + \varepsilon_k (N)) . \]
\end{lem}

\begin{proof} Note that only finitely many factors are different from $1$. Let $N_k=\sum_{\ell=k}^{\infty} b_{\ell} q_{\ell}$. Then $N=N_0 
\ge N_1 \ge N_2 \ge \cdots$, and $N_k=0$ for all large enough $k$. By the definition of Sudler products,
\[ \begin{split} P_N (\alpha ) &= \prod_{k=0}^{\infty} \prod_{n=N_{k+1}+1}^{N_k} f(n \alpha ) \\ &=\prod_{k=0}^{\infty} \prod_{n=1}^{b_k q_k} f(n\alpha +N_{k+1}\alpha ) \\ &= \prod_{k=0}^{\infty} \prod_{b=0}^{b_k-1} \prod_{n=1}^{q_k} f \left( n\alpha +b q_k \alpha + \sum_{\ell =k+1}^{\infty} b_{\ell} q_{\ell} \alpha \right) \\ &= \prod_{k=0}^{\infty} \prod_{b=0}^{b_k-1} \prod_{n=1}^{q_k} f\left( n \alpha + (-1)^k b \delta_k + \sum_{\ell =k+1}^{\infty} (-1)^{\ell} b_{\ell} \delta_{\ell} \right) \\ &= \prod_{k=0}^{\infty} \prod_{b=0}^{b_k-1} P_{q_k} (\alpha, b q_k \delta_k + \varepsilon_k (N)) . \end{split} \]
\end{proof}

The main message of the next lemma is that for quadratic irrational $\alpha$, the function $P_{q_k}(\alpha, x)$ has a positive lower bound at all points which appear in the claim of Lemma \ref{ostrowskilemma}. From now on let $1 \le [k] \le p$ denote the remainder of $k-s$ modulo $p$, where $s$ is the length of the pre-period in the continued fraction for $\alpha$; that is, if $k=s+mp+r$, then $[k]=r$.
\begin{lem}\label{intervallemma} \hspace{1mm}
\begin{enumerate}
\item[(i)] For any integer $N \ge 0$ with Ostrowski expansion $N=\sum_{k=0}^{\infty} b_k q_k$, any $k \ge 0$ and any $0 \le b<b_k$ we have $P_{q_k}(\alpha , bq_k \delta_k + \varepsilon_k (N)) \gg 1$ uniformly in $N$, $k$ and $b$.
\item[(ii)] There exist compact intervals $I_r$, $1 \le r \le p$, and a constant $k_0>s$ with the following properties. First, for any integer $N \ge 0$ with Ostrowski expansion $N=\sum_{k=0}^{\infty} b_k q_k$, any $k \ge k_0$ and any $0 \le b<b_k$ we have $bq_k \delta_k + \varepsilon_k (N) \in I_{[k]}$. Second, $P_{q_k}(\alpha, x) \gg 1$ on $I_{[k]}$ uniformly in $k \ge k_0$.
\end{enumerate}
\end{lem}

\begin{proof} Fix an integer $N \ge 0$ with Ostrowski expansion $\sum_{k=0}^{\infty} b_k q_k$, and integers $k \ge 0$ and $0 \le b <b_k$. We necessarily have $b_k>0$; in particular, $k \ge 1$, or $k=0$ and $a_1>1$. Observe that
\[ \begin{split} \varepsilon_k (N) = q_k \sum_{\ell =k+1}^{\infty} (-1)^{k+\ell} b_{\ell} \delta_{\ell} &\le q_k \left( a_{k+3} \delta_{k+2}+a_{k+5} \delta_{k+4} +\cdots \right) \\&= q_k \left( (\delta_{k+1}-\delta_{k+3} ) + (\delta_{k+3}-\delta_{k+5})+\cdots \right) \\ &= q_k \delta_{k+1} . \end{split} \]
Since $b_k>0$, by the extra rule of the Ostrowski expansions we have $b_{k+1}<a_{k+2}$. Therefore
\[ \begin{split} \varepsilon_k (N) = q_k \sum_{\ell =k+1}^{\infty} (-1)^{k+\ell} b_{\ell} \delta_{\ell} &\ge -q_k \left( (a_{k+2}-1) \delta_{k+1} + a_{k+4} \delta_{k+3} + \cdots  \right) \\ &= -q_k \left( -\delta_{k+1} + (\delta_k - \delta_{k+2}) + (\delta_{k+2}-\delta_{k+4}) + \cdots \right) \\ &=-q_k (\delta_k - \delta_{k+1}) . \end{split} \]
Letting $\kappa >0$ be as in Lemma \ref{deltaklemma}, we thus have $|\varepsilon_k (N)| \le (1-\kappa) q_k \delta_k$.

Consider now
\[ P_{q_k} (\alpha, bq_k \delta_k +\varepsilon_k (N)) = \prod_{n=1}^{q_k} f((n+bq_k)\alpha + (-1)^k \varepsilon_k (N)/q_k) . \]
For each $1 \le n \le q_k$ we have $n+bq_k \le a_{k+1}q_k<q_{k+1}$, and hence by the best approximation property of continued fractions,
\[ \left\| (n+bq_k)\alpha +(-1)^k \varepsilon_k (N) /q_k \right\| \ge \| q_k \alpha \| - |\varepsilon_k (N)|/q_k \ge \kappa \delta_k . \]
Consequently, $P_{q_k}(\alpha, bq_k \delta_k+\varepsilon_k (N)) \gg 1$ for any \textit{fixed} $k \ge 0$. It will thus be enough to prove Lemma \ref{intervallemma} (ii), and Lemma \ref{intervallemma} (i) will follow.

We now prove Lemma \ref{intervallemma} (ii). Observe that \eqref{recursionsolutions} implies $q_{s+mp+r} \delta_{s+mp+r} \to B_r$ as $m \to \infty$, where $B_r=C_r E_r>0$, $1 \le r \le p$, are constants. Define
\[ I_r = \left[ -(1-\kappa /2) B_r, (a_{s+r+1}-\kappa /2) B_r \right] \qquad (1 \le r \le p) . \]
These intervals, together with some constant $k_0$, to be chosen, satisfy the claim. Choosing $k_0$ large enough, for all $k \ge k_0$ and all $0 \le b<a_{k+1}$ we have $bq_k \delta_k +[-(1-\kappa )q_k \delta_k, (1-\kappa )q_k \delta_k] \subseteq I_{[k]}$. In particular, $bq_k \delta_k +\varepsilon_k (N) \in I_{[k]}$ for all $N \ge 0$, all $k \ge k_0$ and all $0 \le b<b_k$.

Now let $k \ge k_0$ and $x \in I_{[k]}$ be arbitrary, and let us prove a lower bound for $P_{q_k}(\alpha, x)$. Then $x=b B_{[k]}+y$ for some appropriate integer $0 \le b<a_{k+1}$, and some $|y| \le (1-\kappa /2)B_{[k]}$. Let
\[ z= \frac{(-1)^k (y+b(B_{[k]} - q_k \delta_k))}{q_k} , \]
and note that
\begin{equation}\label{zestimate}
|z| \le \frac{(1-\kappa /2)B_{[k]} + (a_{k+1}-1)|q_k \delta_k -B_{[k]}|}{q_k} \le (1-\kappa /4) \delta_k
\end{equation}
provided $k_0$ was chosen large enough. With this choice of $z$ we have 
$$
f (n \alpha + (-1)^k x/q_k) = f((n+bq_k)\alpha +z),
$$ 
and so
\begin{equation}\label{pqkfactorization}
\begin{split} \frac{P_{q_k}(\alpha, x)}{\prod_{n=bq_k+1}^{(b+1)q_k} f(n \alpha )} &= \frac{\prod_{n=1}^{q_k} f((n+bq_k)\alpha +z)}{\prod_{n=1}^{q_k} f((n+bq_k)\alpha )} \\ &= \left| \prod_{n=1}^{q_k} \left( \cos (\pi z) + \sin (\pi z) \cot (\pi (n+bq_k)\alpha ) \right) \right|, \end{split}
\end{equation}
where the last equation follows from standard trigonometric identities. Using $\| (n+bq_k) \alpha \| \ge \| q_k \alpha \|=\delta_k$ and \eqref{zestimate}, we obtain
\[ \begin{split} \left| \cos (\pi z)-1 + \sin (\pi z) \cot (\pi (n+bq_k)\alpha ) \right| &\le |\cos (\pi z)-1| + \frac{|\sin (\pi z)|}{\sin (\pi \delta_k )} \\ &\le \frac{\pi^2}{2} (1-\kappa /4)^2 \delta_k^2 + \frac{\pi (1-\kappa /4)\delta_k}{\pi \delta_k - \pi^3 \delta_k^3/6} \\ &\le 1-\kappa /8 \end{split} \]
provided $k_0$ was chosen large enough; the point is that each factor in \eqref{pqkfactorization} is bounded away from $0$. Following the steps in the proof of Proposition \ref{transferprinciple} (in particular, recalling the cotangent sum estimate \eqref{cotangentsumirrational}), we thus deduce that
\[ \begin{split} &\frac{P_{q_k}(\alpha, x)}{\prod_{n=bq_k+1}^{(b+1)q_k} f(n \alpha )} \\ &= \exp \left( O \left( 1+ \delta_k \left| \sum_{n=1}^{q_k} \cot (\pi (n+bq_k)\alpha ) \right| + \delta_k^2 \sum_{n=1}^{q_k} \cot^2 (\pi (n+bq_k)\alpha ) \right) \right) \\ &= \exp \left( O(1) \right) . \end{split} \]
On the other hand, a general result of Lubinsky \cite[Proposition 5.1]{LU} implies that $1 \ll P_N (\alpha) \ll 1$ whenever the Ostrowski expansion of $N$ contains $\ll 1$ nonzero terms. In particular,
\[ \prod_{n=bq_k+1}^{(b+1)q_k} f(n \alpha ) = \frac{P_{(b+1)q_k}(\alpha )}{P_{bq_k}(\alpha)} \gg 1, \]
and hence $P_{q_k} (\alpha, x) \gg 1$, as claimed.
\end{proof}

\subsection{The limit functions}

The perturbed Sudler products $P_{q_k}(\alpha, x)$ were shown to converge to an explicitly given limit function for the special irrationals $\alpha=[0;a,a,a,\dots]$ in \cite{ATZ}. A generalization to all quadratic irrationals has recently been announced by Grepstad, Technau and Zafeiropoulos \cite{GNZ}; see also \cite{GN} for a version without the perturbation variable $x$. In this paper we prove the locally uniform convergence of $P_{q_k}(\alpha, x)$ with an explicit rate. This explicit (in fact, exponential) rate is needed to derive the $O(\max \{ 1,1/c\})$ resp.\ $O(1)$ error terms in Theorem \ref{quadraticasymptotics}.

Let $\alpha=[a_0;a_1,\ldots, a_s, \overline{a_{s+1},\ldots, a_{s+p}}]$ be a quadratic irrational, and let $C_r, E_r$ be as in \eqref{recursionsolutions}. For any $1 \le r \le p$ let
\begin{equation}\label{Grdef}
G_r(\alpha, x)= 2 \pi |x+C_r E_r| \prod_{n=1}^{\infty} \left| \left( 1 - C_r E_r \frac{\{ n \alpha_r \} -\frac{1}{2}}{n} \right)^2 - \frac{\left( x+\frac{C_r E_r}{2} \right)^2}{n^2} \right| ,
\end{equation}
where
\[ \alpha_r = [0;\overline{a_{s+r+p}, \ldots, a_{s+r+2},a_{s+r+1}}] . \]

\begin{thm}\label{limittheorem} The infinite product in \eqref{Grdef} is locally uniformly convergent on $\mathbb{R}$. The function $G_r(\alpha , \cdot )$ is continuous on $\mathbb{R}$, and continuously differentiable on the open set $\{ x \in \mathbb{R} \, : \, G_r (\alpha, x) \neq 0 \}$. For any compact interval $I \subset \mathbb{R}$ and any integer $k \ge 1$,
\[ P_{q_k} (\alpha, x) = \left( 1 + O \left( q_k^{-1/2} \log^{3/4} q_k \right) \right) G_{[k]} (\alpha, x) + O \left( q_k^{-2} \right) , \qquad x \in I \]
with implied constants depending only on $I$ and $\alpha$. In particular, $P_{q_{s+mp+r}} (\alpha, \cdot ) \to G_r (\alpha, \cdot)$ locally uniformly on $\mathbb{R}$, as $m \to \infty$.
\end{thm}
\noindent We postpone the proof to Section \ref{limitfunctionsection}. Note that the periodicity of the continued fraction expansion is crucial for such a limit relation; in particular, Theorem \ref{limittheorem} does not hold for all badly approximable $\alpha$.

Lemma \ref{intervallemma} implies that $G_r (\alpha, \cdot )>0$, and consequently that $\log G_r (\alpha, \cdot )$ is Lipschitz on the compact interval $I_r$. The main idea of the proof of Theorem \ref{quadraticasymptotics} is to replace the perturbed Sudler product $P_{q_k} (\alpha, bq_k \delta_k + \varepsilon_k (N))$ by its limit $G_{[k]}(\alpha, bq_k \delta_k + \varepsilon_k (N))$ in the claim of Lemma \ref{ostrowskilemma}. To this end, for any integer $N \ge 0$ with Ostrowski expansion $N=\sum_{k=0}^{\infty} b_k q_k$ let
\[ G_N (\alpha ) = \prod_{k=k_0}^{\infty} \prod_{b=0}^{b_k-1} G_{[k]} (\alpha, bq_k \delta_k + \varepsilon_k (N)), \]
where $k_0$ is the constant from the conclusion of Lemma \ref{intervallemma} (ii).
\begin{lem}\label{pngnlemma} For any real $c>0$ and any integer $\ell \ge 1$, we have
\[ \log \left( \sum_{0 \le N < q_{\ell}} P_N(\alpha )^c \right)^{1/c} = \log \left( \sum_{0 \le N<q_{\ell}} G_N (\alpha )^c \right)^{1/c} + O(1), \]
as well as
\[ \log \max_{0 \le N <q_{\ell}} P_N (\alpha ) = \log \max_{0 \le N <q_{\ell}} G_N (\alpha ) + O(1). \]
\end{lem}

\begin{proof} By Lemma \ref{ostrowskilemma}, for all $0 \le N <q_{\ell}$ with Ostrowski expansion $N=\sum_{k=0}^{\ell-1} b_k q_k$,
\[ \frac{P_N (\alpha )}{G_N (\alpha )} = \prod_{k=0}^{k_0-1} P_{q_k} (\alpha, bq_k \delta_k + \varepsilon_k (N)) \cdot \prod_{k=k_0}^{\ell-1} \prod_{b=0}^{b_k-1} \frac{P_{q_k}(\alpha, bq_k \delta_k + \varepsilon_k (N))}{G_{[k]}(\alpha, bq_k \delta_k + \varepsilon_k (N))} . \]
Lemma \ref{intervallemma} (i) shows that here $1 \ll \prod_{k=0}^{k_0-1} P_{q_k} (\alpha, bq_k \delta_k + \varepsilon_k (N)) \ll 1$. Since by Lemma \ref{intervallemma} (ii) we have $G_{[k]} (\alpha, b q_k \delta_k + \varepsilon_k (N)) \gg 1$, each factor $k_0 \le k \le \ell -1$ also stays between two positive constants. Finally, for a large enough constant $k_1>k_0$ Theorem \ref{limittheorem} gives
\[ \begin{split} \prod_{k=k_1}^{\ell-1} \prod_{b=0}^{b_k-1} \frac{P_{q_k}(\alpha, bq_k \delta_k + \varepsilon_k (N))}{G_{[k]}(\alpha, bq_k \delta_k + \varepsilon_k (N))} &= \prod_{k=k_1}^{\ell-1} \prod_{b=0}^{b_k-1} \left( 1 + O \left( q_k^{-1/2} \log^{3/4} q_k \right) \right) \\ &= \prod_{k=k_1}^{\infty} \left( 1 + O \left( q_k^{-1/2} \log^{3/4} q_k \right) \right) \in [1/2,2] . \end{split} \]
Hence $1 \ll P_N (\alpha ) / G_N (\alpha ) \ll 1$, and the claims follow.
\end{proof}

\subsection{Approximate additivity}

The final step is to prove that our sequences with $P_N(\alpha)$ replaced by $G_N(\alpha)$ are additive up to a small error; the proof of Theorem \ref{quadraticasymptotics} will then be immediate.
\begin{lem}\label{additivelemma} For any real $c>0$, the sequences
\[ c_m := \log \left( \sum_{0 \le N<q_{k_0+mp}} G_N (\alpha )^c \right)^{1/c} \qquad \textrm{and} \qquad c_m^* := \log \max_{0 \le N<q_{k_0+mp}} G_N (\alpha ) \]
satisfy $c_{m+n} = c_m+c_n+O(\max \{ 1,1/c \})$ and $c_{m+n}^* = c_m^*+c_n^*+O(1)$ for all $m,n \ge 1$.
\end{lem}

\begin{proof} First, we prove that $c_m$ and $c_m^*$ are approximately subadditive. Let $0 \le N < q_{k_0+(m+n)p}$ be an integer with Ostrowski expansion $N=\sum_{k=0}^{k_0+(m+n)p-1} b_k q_k$. Consider the natural factorization
\begin{equation}\label{gnfactorization}
G_N (\alpha ) = \prod_{k=k_0}^{k_0+mp-1} \prod_{b=0}^{b_k-1} G_{[k]}(\alpha, b q_k \delta_k + \varepsilon_k (N)) \prod_{k=k_0+mp}^{k_0+(m+n)p-1} \prod_{b=0}^{b_k-1} G_{[k]}(\alpha, b q_k \delta_k + \varepsilon_k (N)) .
\end{equation}
Let us write $N=N_1+N_2$ with $N_1=\sum_{k=0}^{k_0+mp-1} b_k q_k$ and $N_2=\sum_{k=k_0+mp}^{k_0+(m+n)p-1} b_k q_k$. The plan of the proof is to show that, making a small error, we can replace $\varepsilon_k (N)$ in \eqref{gnfactorization} by $\varepsilon_k (N_1)$ and $\varepsilon_k (N_2)$, respectively, so that $G_N \approx G_{N_1}G_{N_2}$. Then we will show that we can replace $N_2$ by a number $N_2^*$ having the ``shifted'' Ostrowski representation $N_2^*=\sum_{k=k_0}^{k_0+np-1} b_{k+mp} q_k$, and obtain $G_N \approx G_{N_1} G_{N_2^*}$. This approximate shift-invariance is crucial for the argument, and comes from the periodicity of the continued fraction expansion of $\alpha$. From $G_N \approx G_{N_1} G_{N_2^*}$ we can deduce that $c_{m+n} \approx c_m + c_n$ and $c_{m+n}^* \approx c_m^* + c_n^*$, which is what we want to prove.

Now we make this precise. Note that for any $k_0 \le k \le k_0+mp-1$,
\[ \begin{split} \left| \varepsilon_k (N) - \varepsilon_k (N_1) \right| &= \left| q_k \sum_{\ell =k+1}^{k_0+(m+n)p-1} (-1)^{k+\ell} b_{\ell} \delta_{\ell} - q_k \sum_{\ell =k+1}^{k_0+mp-1} (-1)^{k+\ell} b_{\ell} \delta_{\ell} \right| \\ &\ll q_k \sum_{\ell =k_0+mp}^{k_0+(m+n)p-1} \delta_{\ell} \\ &\ll q_k \delta_{k_0+mp} .  \end{split} \]
Since $\log G_r$ is Lipschitz on $I_r$, the previous estimate implies that the first factor in \eqref{gnfactorization} is
\begin{equation}\label{gn1}
\begin{split} \prod_{k=k_0}^{k_0+mp-1} \prod_{b=0}^{b_k-1} &G_{[k]}(\alpha, b q_k \delta_k + \varepsilon_k (N)) \\ &= \prod_{k=k_0}^{k_0+mp-1} \prod_{b=0}^{b_k-1} G_{[k]}(\alpha, b q_k \delta_k + \varepsilon_k (N_1)) e^{O(q_k \delta_{k_0+mp})} \\ &= e^{O(q_{k_0+mp} \delta_{k_0+mp})} \prod_{k=k_0}^{k_0+mp-1} \prod_{b=0}^{b_k-1} G_{[k]}(\alpha, b q_k \delta_k + \varepsilon_k (N_1)) \\ &= e^{O(1)} G_{N_1}(\alpha ) . \end{split}
\end{equation}

Now let $N_2^*=\sum_{k=k_0}^{k_0+np-1} b_{k+mp}q_k$; observe that this is a valid Ostrowski expansion of $0 \le N_2^* < q_{k_0+np}$. Using \eqref{recursionsolutions}, for any $k_0+mp \le k \le k_0+(m+n)p-1$ and any $0 \le b < b_k$,
\[ \left| b q_k \delta_k - b q_{k-mp} \delta_{k-mp} \right| \ll \eta^{2m-2k/p} .  \]
Similarly,
\[ \begin{split} |\varepsilon_k (N) &- \varepsilon_{k-mp} (N_2^*)| \\ &= \left| q_k \sum_{\ell =k+1}^{k_0+(m+n)p-1} (-1)^{k+\ell} b_{\ell} \delta_{\ell} - q_{k-mp} \sum_{\ell =k-mp+1}^{k_0+np-1} (-1)^{k-mp+\ell} b_{\ell +mp} \delta_{\ell} \right| \\ &=\left| \sum_{\ell =k+1}^{k_0+(m+n)p-1} (-1)^{k+\ell} b_{\ell} \left( q_k \delta_{\ell} - q_{k-mp} \delta_{\ell -mp} \right) \right| \\ &\ll \sum_{\ell =k+1}^{k_0+(m+n)p-1} \eta^{2m-k/p-\ell /p} \\ &\ll \eta^{2m-2k/p} . \end{split} \]
Therefore the second factor in \eqref{gnfactorization} is
\begin{equation}\label{gn2}
\begin{split} \prod_{k=k_0+mp}^{k_0+(m+n)p-1} \prod_{b=0}^{b_k-1} &G_{[k]}(\alpha, b q_k \delta_k + \varepsilon_k (N)) \\ &= \prod_{k=k_0+mp}^{k_0+(m+n)p-1} \prod_{b=0}^{b_k-1} G_{[k]}(\alpha, b q_{k-mp} \delta_{k-mp} + \varepsilon_{k-mp} (N_2^*)) e^{O(\eta^{2m-2k/p})} \\ &=e^{O(1)} \prod_{k=k_0}^{k_0+np-1} \prod_{b=0}^{b_{k+mp}-1} G_{[k]} (\alpha, b q_k \delta_k +\varepsilon_k (N_2^*) ) \\ &= e^{O(1)} G_{N_2^*} (\alpha ) . \end{split}
\end{equation}

From \eqref{gnfactorization}, \eqref{gn1} and \eqref{gn2} we finally obtain the approximate factorization $G_N(\alpha) = G_{N_1}(\alpha) G_{N_2^*} (\alpha) e^{O(1)}$. As $N$ runs in the interval $0 \le N < q_{k_0+(m+n)p}$, we obtain each pair $(N_1, N_2^*) \in [0,q_{k_0+mp}) \times [0,q_{k_0+np})$ at most once by the uniqueness of Ostrowski expansions. Therefore
\[ \begin{split} \sum_{0 \le N <q_{k_0+(m+n)p}} G_N (\alpha )^c &\le \left( \sum_{0 \le N<q_{k_0+mp}} G_N (\alpha )^c \right) \left( \sum_{0 \le N<q_{k_0+np}} G_N (\alpha )^c \right) e^{O(c)}, \\ \max_{0 \le N <q_{k_0+(m+n)p}} G_N (\alpha ) &\le \left( \max_{0 \le N<q_{k_0+mp}} G_N (\alpha ) \right) \left( \max_{0 \le N<q_{k_0+np}} G_N (\alpha ) \right) e^{O(1)}, \end{split} \]
and the approximate subadditivity $c_{m+n} \le c_m+c_n +O(1)$ and $c_{m+n}^* \le c_m^*+c_n^*+O(1)$ follows.

The proof of approximate superadditivity is entirely analogous. Let $0 \le N' < q_{k_0+mp}$ and $0 \le N'' < q_{k_0+np}$ be integers with Ostrowski expansions $N'=\sum_{k=0}^{k_0+mp-1} b_k' q_k$ and $N''=\sum_{k=0}^{k_0+np-1} b_k'' q_k$. Define $N=\sum_{k=0}^{k_0+(m+n+1)p-1} b_k q_k$, where
\[ b_k= \left\{ \begin{array}{ll} b_k' & \textrm{if } 0 \le k \le k_0 +mp-1, \\ 0 & \textrm{if } k_0+mp \le k \le k_0 +(m+1)p-1, \\ b_{k-(m+1)p}'' & \textrm{if } k_0+(m+1)p \le k \le k_0+(m+n+1)p-1 . \end{array} \right. \]
Note that this is a valid Ostrowski expansion of an integer $0 \le N < q_{k_0+(m+n+1)p}$; the block of zeroes in the middle ensures that the extra rule ($b_{k-1}=0$ whenever $b_k=a_{k+1}$) is satisfied.

Repeating the arguments from above, we deduce the approximate factorization $G_N(\alpha) =G_{N'}(\alpha) G_{N''}(\alpha)e^{O(1)}$. Observe that as $(N',N'') \in [0,q_{k_0+mp}) \times [0,q_{k_0+np})$, we obtain each integer $N \in [0,q_{k_0+(m+n+1)p})$ at most $O(1)$ times; indeed, from the value of $N$ one can recover all Ostrowski digits of $N'$ and $N''$ except for $b_k''$, $0 \le k \le k_0-1$. Therefore
\[ \begin{split} \left( \sum_{0 \le N<q_{k_0+mp}} G_N (\alpha )^c \right) \left( \sum_{0 \le N<q_{k_0+np}} G_N (\alpha )^c \right) &\le \sum_{0 \le 
N <q_{k_0+(m+n+1)p}} G_N (\alpha )^c e^{O(c)}, \\ \left( \max_{0 \le N<q_{k_0+mp}} G_N (\alpha ) \right) \left( \max_{0 \le N<q_{k_0+np}} G_N (\alpha ) \right) &\le \max_{0 \le N <q_{k_0+(m+n+1)p}} G_N (\alpha ) e^{O(1)}, \end{split} \]
and we get $c_m+c_n \le c_{m+n+1}+O(1)$ and $c_m^*+c_n^* \le c_{m+n+1}^*+O(1)$. By the approximate subadditivity proved above, here
\[ \begin{split} c_{m+n+1} &\le c_{m+n}+c_1+O(1) = c_{m+n}+O(\max \{ 1,1/c \} ), \\ c_{m+n+1}^* &\le c_{m+n}^*+c_1^*+O(1)=c_{m+n}^*+O(1). \end{split} \] 
Hence
\[ \begin{split} c_m+c_n-O(\max \{ 1,1/c \}) &\le c_{m+n} \le c_m+c_n+O(1), \\ c_m^*+c_n^*-O(1) &\le c_{m+n}^* \le c_m^*+c_n^*+O(1), \end{split} \]
as claimed.
\end{proof}

\begin{proof}[Proof of Theorem \ref{quadraticasymptotics}] According to Lemma \ref{additivelemma}, the sequence $c_m+L$ is subadditive, and the sequence $c_m-L$ is superadditive with some constant $L=O(\max \{ 1,1/c \})$. Fekete's subadditive lemma thus shows that $c_m/m$ converges, and its limit is
\[ p K_c(\alpha ) :=\lim_{m \to \infty} \frac{c_m}{m} = \inf_{m \ge 1} \frac{c_m+L}{m} = \sup_{m \ge 1} \frac{c_m-L}{m} . \]
The previous relations also yield the rate of convergence $c_m=K_c(\alpha ) mp+O(\max \{ 1 , 1/c\})$. Since $K_c (\alpha ) = O (\max \{ 1,1/c \})$ (see \eqref{KcKinftybounds}), we have
\[ \log \left( \sum_{0 \le N < q_k} G_N (\alpha )^c \right)^{1/c} = K_c (\alpha ) k + O(\max \{ 1,1/c \} ). \]
Lemma \ref{pngnlemma} and Proposition \ref{transferprinciple} show that here $G_N (\alpha )$ can be replaced by $P_N(\alpha )$ and also by $P_N (p_k/q_k )$. Hence
\[ \log \left( \sum_{0 \le N < q_k} P_N (p_k/q_k)^c \right)^{1/c} = K_c (\alpha ) k + O(\max \{ 1,1/c \} ) , \]
as claimed. An identical proof gives
\[ \log \max_{0 \le N < q_k} P_N (p_k/q_k) = K_{\infty} (\alpha ) k + O(1). \]
It follows from \eqref{triviallowerbound} and \eqref{KcKinftybounds} that the constants $K_c (\alpha)$ and $K_{\infty} (\alpha)$ are positive.
\end{proof}

\section{Locally uniform convergence of perturbed Sudler products}\label{limitfunctionsection}

\begin{proof}[Proof of Theorem \ref{limittheorem}] Fix a compact interval $I \subseteq \mathbb{R}$. Throughout the proof, constants and implied constants depend only on $I$ and $\alpha$. For the sake of readability, we continue to write $f(x)=|2 \sin (\pi x)|$.

We start by peeling off the last factor in $P_{q_k}(\alpha ,x)$ and using $\alpha = p_k/q_k+(-1)^k \delta_k/q_k$ to get
\[ \begin{split} P_{q_k}(\alpha, x) &= f (q_k \alpha + (-1)^k x/q_k) \prod_{n=1}^{q_k-1} f \left( \frac{np_k}{q_k} + (-1)^k \frac{n\delta_k +x}{q_k} \right) \\ &= f(\delta_k +x/q_k) \prod_{n=1}^{q_k-1} f \left( \frac{np_k}{q_k} + (-1)^k \left( \left\{ \frac{n}{q_k} \right\} -\frac{1}{2} \right) \delta_k + (-1)^k \frac{2x+q_k\delta_k}{2q_k} \right) . \end{split} \]
The factors in the previous line depend only on the remainder of $n$ modulo $q_k$. The general identity $q_k p_{k-1}-p_k q_{k-1}=(-1)^k$ in the theory of continued fractions shows that $q_k$ and $q_{k-1}$ are relatively prime, and $p_k q_{k-1} \equiv (-1)^{k+1} \pmod{q_k}$. Reordering the product via the bijection $n \mapsto q_{k-1}n$ of the set of nonzero residues modulo $q_k$, and using the identity \eqref{lastterm}, we obtain
\[ P_{q_k}(\alpha, x) = f( \delta_k +x/q_k ) q_k \prod_{n=1}^{q_k-1} \frac{f \left( \frac{n}{q_k} - \left( \left\{ \frac{nq_{k-1}}{q_k} \right\} - \frac{1}{2} \right) \delta_k - \frac{2x+q_k \delta_k}{2q_k} \right)}{f(n/q_k)} . \]
Let us combine the $n$th and the $(q_k-n)$th factors in the product via the trigonometric identity $f(u-v)f(u+v)=|f^2(u)-f^2(v)|$. If $q_k$ is even, then the $n=q_k/2$ factor is $1+O(q_k^{-2})$, thus
\begin{equation}\label{pqkproduct}
\begin{split} P_{q_k}(\alpha ,x) = &(1+O(q_k^{-2})) f(\delta_k +x/q_k)q_k \\ &\times \prod_{0<n<q_k/2} \frac{\left| f^2 \left( \frac{n}{q_k} - \left( \left\{ \frac{nq_{k-1}}{q_k} \right\} - \frac{1}{2} \right) \delta_k \right) - f^2 \left( \frac{2x+q_k \delta_k}{2q_k} \right) \right|}{f^2 \left( \frac{n}{q_k} \right)} . \end{split}
\end{equation}

Let $\log t \ll \psi (t) =o(t)$ be a function, to be chosen. We now show that the factors $\psi(q_k) \le n < q_k/2$ in \eqref{pqkproduct} have negligible contribution. Simple trigonometric identities and estimates give
\[ \begin{split} &\frac{\left| f^2 \left( \frac{n}{q_k} - \left( \left\{ \frac{nq_{k-1}}{q_k} \right\} - \frac{1}{2} \right) \delta_k \right) - f^2 \left( \frac{2x+q_k \delta_k}{2q_k} \right) \right|}{f^2 \left( \frac{n}{q_k} \right)} \\ &\hspace{20mm}= \frac{f^2 \left( \frac{n}{q_k} - \left( \left\{ \frac{nq_{k-1}}{q_k} \right\} - \frac{1}{2} \right) \delta_k \right)}{f^2 \left( \frac{n}{q_k} \right)} + O \left( \frac{1}{n^2} \right) \\ &\hspace{20mm} = 1 - \sin \left( 2 \pi \left( \left\{ \frac{nq_{k-1}}{q_k} \right\} - \frac{1}{2} \right) \delta_k \right) \cot \left( \pi \frac{n}{q_k} \right) + O \left( \frac{1}{n^2} \right) .  \end{split} \]
In particular, each factor $\psi (q_k) \le n < q_k/2$ is $1+O(1/n)$, and thus
\[ \begin{split} &\prod_{\psi (q_k) \le n < q_k/2} \frac{\left| f^2 \left( \frac{n}{q_k} - \left( \left\{ \frac{nq_{k-1}}{q_k} \right\} - \frac{1}{2} \right) \delta_k \right) - f^2 \left( \frac{2x+q_k \delta_k}{2q_k} \right) \right|}{f^2 \left( \frac{n}{q_k} \right)} \\ &= \exp \left( O \left( \left| \sum_{\psi (q_k) \le n < q_k/2} \sin \left( 2 \pi \left( \left\{ \frac{nq_{k-1}}{q_k} \right\} - \frac{1}{2} \right) \delta_k \right) \cot \left( \pi \frac{n}{q_k} \right)  \right| + \frac{1}{\psi (q_k)} \right) \right) . \end{split} \]
Since the partial quotients of $q_{k-1}/q_k=[0;a_k,a_{k-1}, \dots, a_1]$ are bounded by a constant depending only on $\alpha$, by a classical estimate \cite[pp.\ 125]{KN} the discrepancy of the sequence $\{ n q_{k-1}/q_k \}$, $1 \le n \le N$ is $\ll (\log N)/N$ provided that\footnote{Such discrepancy estimates are usually stated for irrational numbers, but the same proof works for rational numbers as well provided that $N$ is less than the denominator.} $N<q_k$. Applying Koksma's inequality \cite[pp.\ 143]{KN} to the mean zero, $1$-periodic function $\sin (2 \pi (\{ t \}-1/2)\delta_k)$ of total variation $\ll q_k^{-1}$, we thus deduce
\[ \left| \sum_{n=1}^N \sin \left( 2 \pi \left( \left\{ \frac{nq_{k-1}}{q_k} \right\} - \frac{1}{2} \right) \delta_k \right) \right| \ll \frac{\log N}{q_k} \qquad (1 \le N <q_k/2) . \]
Observe also that $\cot (\pi n /q_k)$ is monotone in $0<n<q_k/2$. Summation by parts yields
\[ \left| \sum_{\psi (q_k) \le n < q_k/2} \sin \left( 2 \pi \left( \left\{ \frac{nq_{k-1}}{q_k} \right\} - \frac{1}{2} \right) \delta_k \right) \cot \left( \pi \frac{n}{q_k} \right)  \right| \ll \frac{\log q_k}{\psi (q_k)} , \]
hence
\begin{equation}\label{psiqk<n<qk/2}
\prod_{\psi (q_k) \le n < q_k/2} \frac{\left| f^2 \left( \frac{n}{q_k} - \left( \left\{ \frac{nq_{k-1}}{q_k} \right\} - \frac{1}{2} \right) \delta_k \right) - f^2 \left( \frac{2x+q_k \delta_k}{2q_k} \right) \right|}{f^2 \left( \frac{n}{q_k} \right)} = 1+O \left( \frac{\log q_k}{\psi (q_k)} \right) .
\end{equation}

Next, fix a large constant $N_0>0$, and consider the factors $N_0<n <\psi(q_k)$ in \eqref{pqkproduct}. Applying the estimate $f^2(t)=4 \pi^2 t^2 (1+O(t^2))$ in all three terms introduces a multiplicative error of $\prod_{N_0<n<\psi (q_k)} \left( 1+O(n^2/q_k^2) \right) = 1+O(\psi(q_k)^3/q_k^2)$, and we get
\[ \begin{split} &\prod_{N_0<n<\psi (q_k)} \frac{\left| f^2 \left( \frac{n}{q_k} - \left( \left\{ \frac{nq_{k-1}}{q_k} \right\} - \frac{1}{2} \right) \delta_k \right) - f^2 \left( \frac{2x+q_k \delta_k}{2q_k} \right) \right|}{f^2 \left( \frac{n}{q_k} \right)} \\ &= \left( 1 + O \left( \frac{\psi(q_k)^3}{q_k^2} \right) \right) \prod_{N_0<n<\psi(q_k)} \left| \left( 1-q_k \delta_k \frac{\left\{ \frac{nq_{k-1}}{q_k} \right\} - \frac{1}{2}}{n} \right)^2 - \frac{\left( x+ \frac{q_k \delta_k}{2} \right)^2}{n^2} \right| . \end{split} \]
Choosing $N_0$ large enough, we can ensure that each factor in the previous product stays in, say, $[1/2,2]$. By the explicit formula \eqref{recursionsolutions} for $q_k$ and $\delta_k$, we have $q_k \delta_k=C_{[k]}E_{[k]}+O(q_k^{-2})$. Hence replacing $q_k \delta_k$ by $C_{[k]}E_{[k]}$ in the previous formula introduces a negligible multiplicative error of $\prod_{N_0<n<\psi (q_k)} (1+O(q_k^{-2}n^{-1}))=1+O(q_k^{-2}\log \psi (q_k))$.

We also wish to replace $q_{k-1}/q_k=[0;a_k,a_{k-1},\dots, a_1]$ by its limit $\alpha_{[k]}$; recall that $\alpha_r=[0;\overline{a_{s+r+p}, \dots, a_{s+r+2},a_{s+r+1}}]$. By the explicit formula \eqref{recursionsolutions}, we have $q_{k-1}/q_k=\alpha_{[k]}+O(q_k^{-2})$ (in particular, $\alpha_r=C_{r-1}/C_r$, $1 \le r \le p$ with the convention $C_0=C_p/\eta$). Observe that the function $(\{ nt \}-1/2)/n$ consists of linear segments of slope 1, with jumps at the points $j/n$, $j \in \mathbb{Z}$. Since $q_{k-1}/q_k$ has bounded partial quotients, for all $N_0<n<\psi (q_k)$ and all $j \in \mathbb{Z}$ we have
\[ \left| \frac{q_{k-1}}{q_k} - \frac{j}{n} \right| \ge \frac{\| n q_{k-1}/q_k \|}{n} \gg \frac{1}{n^2} , \]
hence there is no jump between $q_{k-1}/q_k$ and $\alpha_{[k]}$. Therefore
\[ \left| \frac{\left\{ \frac{n q_{k-1}}{q_k} \right\} - \frac{1}{2}}{n} - \frac{\left\{ n \alpha_{[k]} \right\}\ - \frac{1}{2}}{n} \right| \ll \frac{1}{q_k^2} , \]
so replacing $q_{k-1}/q_k$ by $\alpha_{[k]}$ introduces a negligible multiplicative error of $1+O(q_k^{-2} \psi (q_k))$. We have thus deduced
\begin{equation}\label{N0<n<psiqk}
\begin{split} &\prod_{N_0<n<\psi (q_k)} \frac{\left| f^2 \left( \frac{n}{q_k} - \left( \left\{ \frac{nq_{k-1}}{q_k} \right\} - \frac{1}{2} \right) \delta_k \right) - f^2 \left( \frac{2x+q_k \delta_k}{2q_k} \right) \right|}{f^2 \left( \frac{n}{q_k} \right)} \\ &= \left( 1 + O \left( \frac{\psi(q_k)^3}{q_k^2} \right) \right) \prod_{N_0<n<\psi(q_k)} \left| \left( 1-C_{[k]} E_{[k]} \frac{\{ n \alpha_{[k]} \} - \frac{1}{2}}{n} \right)^2 - \frac{\left( x+ \frac{C_{[k]} E_{[k]}}{2} \right)^2}{n^2} \right| . \end{split}
\end{equation}

\begin{lem}\label{cauchylemma} For any integers $N_0<N_1 \le N_2$ and any $1 \le r \le p$,
\[ \prod_{n=N_1}^{N_2} \left| \left( 1-C_r E_r \frac{\{ n \alpha_r \} - \frac{1}{2}}{n} \right)^2 - \frac{\left( x+ \frac{C_r E_r}{2} \right)^2}{n^2} \right| = 1+O \left( \frac{\log N_1}{N_1} \right) . \]
\end{lem}

\begin{proof} Choosing the constant $N_0$ large enough, each factor stays in, say, $[1/2,2]$. Since the partial quotients of $\alpha_r$ are bounded by a constant depending only on $\alpha$, the same discrepancy estimate and Koksma's inequality yield $\left| \sum_{n=1}^N (\{ n \alpha_r \} -1/2) \right| \ll \log N$ for all $N \ge 1$. Applying summation by parts, we deduce
\[ \left| \sum_{n=N_1}^{N_2} \frac{\{ n \alpha_r \} -\frac{1}{2}}{n} \right| \ll \frac{\log N_1}{N_1}, \]
and the claim of Lemma \ref{cauchylemma} follows.
\end{proof}

Lemma \ref{cauchylemma} immediately implies via the Cauchy criterion that the infinite product
\begin{equation}\label{infiniteproduct}
\prod_{n=N_0+1}^{\infty} \left| \left( 1-C_r E_r \frac{\{ n \alpha_r \} - \frac{1}{2}}{n} \right)^2 - \frac{\left( x+ \frac{C_r E_r}{2} \right)^2}{n^2} \right|
 \end{equation}
is uniformly convergent on $I$, and its limit is positive. Its logarithm is given by a uniformly convergent series; the series of term-by-term derivatives
\[ \sum_{n=N_0+1}^{\infty} \left( \left( 1-C_r E_r \frac{\{ n \alpha_r \} - \frac{1}{2}}{n} \right)^2 - \frac{\left( x+ \frac{C_r E_r}{2} \right)^2}{n^2} \right)^{-1} \frac{-2x-C_r E_r}{n^2} \]
is also seen to be uniformly convergent on $I$. Therefore the logarithm of the infinite product in \eqref{infiniteproduct} is continuously differentiable on $I$; clearly so is the infinite product itself. Multiplying by the missing factors $1 \le n \le N_0$, it follows that the infinite product \eqref{Grdef} defining $G_r(\alpha, x)$ is uniformly convergent on $I$, $G_r (\alpha,x)$ is continuous on $I$, and continuously differentiable on $I$ except at its (finitely many) zeroes. This proves all claims on $G_r(\alpha,x)$.

Repeating the arguments above for the factors $1 \le n \le N_0$, we deduce
\[ \begin{split} &f(\delta_k +x/q_k) q_k \prod_{n=1}^{N_0} \frac{\left| f^2 \left( \frac{n}{q_k} - \left( \left\{ \frac{nq_{k-1}}{q_k} \right\} - \frac{1}{2} \right) \delta_k \right) - f^2 \left( \frac{2x+q_k \delta_k}{2q_k} \right) \right|}{f^2 \left( \frac{n}{q_k} \right)} \\&= 2 \pi |x+C_{[k]} E_{[k]}| \prod_{n=1}^{N_0} \left| \left( 1-C_{[k]} E_{[k]} \frac{\{ n \alpha_{[k]} \} - \frac{1}{2}}{n} \right)^2 - \frac{\left( x+ \frac{C_{[k]} E_{[k]}}{2} \right)^2}{n^2} \right| +O(q_k^{-2}) \end{split} \]
with an additive instead of multiplicative error term, since the factors are not bounded away from zero. Lemma \ref{cauchylemma} also shows that the product on the right hand side of \eqref{N0<n<psiqk} can be extended to all $N_0<n$ up to a negligible error, therefore combining \eqref{pqkproduct}--\eqref{N0<n<psiqk} and the previous formula, we obtain
\[ P_{q_k}(\alpha ,x) = \left( 1+O \left( \frac{\log q_k}{\psi (q_k)} + \frac{\psi (q_k)^3}{q_k^2} \right) \right) G_{[k]}(\alpha ,x) + O(q_k^{-2}). \]
The optimal choice is $\psi (t)=t^{1/2} \log^{1/4} t$. This finishes the proof of Theorem \ref{limittheorem}.
\end{proof}

\section*{Acknowledgements}

CA is supported by the Austrian Science Fund (FWF), projects F-5512, I-3466, I-4945 and Y-901. BB is supported by FWF project Y-901. We want to thank Agamemnon Zafeiropoulos for drawing our attention to the papers of Bettin and Drappeau, and for keeping us informed about his joint work with Grepstad and Neum\"uller.

\end{document}